\RequirePackage[l2tabu,orthodox]{nag}    
\documentclass[a4paper,10pt]{amsart}

\usepackage{etex}
\usepackage[T1]{fontenc}
\usepackage[utf8]{inputenx}
\usepackage{lmodern}
\usepackage[kerning=true,tracking=true]{microtype}
\usepackage{graphics}
\usepackage[pdftex,pagebackref]{hyperref}
\hypersetup{
  colorlinks   = true,  
  urlcolor     = black, 
  linkcolor    = black, 
  citecolor    = black  
}

\title[Connections]%
{Affine connections on complex manifolds of algebraic dimension zero}
\date{June 2, 2015}

\author[S. Dumitrescu]{Sorin Dumitrescu}
\address{Université  Nice Sophia Antipolis,  CNRS,  Laboratoire J.-A. Dieudonn\'e, UMR 7351, 06108 NICE Cedex 02, France}
\email{dumitres@unice.fr}
 
\author[B. McKay]{Benjamin McKay}
\address{University College Cork, Cork, Ireland}
\email{b.mckay@ucc.ie}

\date{\today}

\usepackage{varwidth}
\usepackage{xspace}
\usepackage{amsfonts}
\usepackage{amssymb}
\usepackage{mathtools}
\usepackage{braket}
\usepackage{varioref}
\usepackage{xstring}
\usepackage{array}
\usepackage{ragged2e}
\usepackage{longtable}
\usepackage{multirow}
\usepackage{booktabs}
\usepackage[mathscr]{euscript}
\usepackage{mathrsfs}
\usepackage{tikz-cd}
\usepackage{tensor}
\newtheorem{theorem}{Theorem}

\newtheorem{lemma}{Lemma}
\newtheorem{proposition}{Proposition}
\theoremstyle{remark}

\newtheorem{definition}{Definition}
\newtheorem{example}{Example}

\newcounter{remarkCounter}
\newcommand*{\pr}[1]{\ensuremath{\left(#1\right)}}
\newcommand*{\of}[1]{\ensuremath{\!\pr{#1}}}
\newcommand*{\C}[1]{\ensuremath{\mathbb{C}^{#1}}}
\newcommand*{\Z}[1]{\ensuremath{\mathbb{Z}^{#1}}}
\newcommand*{\GL}[1]{\ensuremath{\operatorname{GL}\of{#1}}}

\newcommand*{\homology}[2]{\ensuremath{H_{#1}\of{#2}}}
\newcommand*{\LieDer}{\ensuremath{\EuScript L}}
\newcommand*{\pd}[2]{\frac{\partial #1}{\partial #2}}
\newcommand*{\defeq}{\mathrel{\vcenter{\baselineskip0.5ex \lineskiplimit0pt
                     \hbox{\scriptsize.}\hbox{\scriptsize.}}}%
                     =}
\newcommand*{\p}[1]{\ensuremath{\partial_{#1}}}
\newcommand*{\Lie}[1]{\ensuremath{\mathfrak{\lowercase{#1}}}}
\newcommand*{\MakeLie}[1]{\expandafter\def\csname Lie#1\endcsname{\Lie{#1}}}
\def\lst{G,H}
\makeatletter
\@for\i:=\lst\do{\expandafter\MakeLie \i}
\makeatother

\makeatletter
\def\@tocline#1#2#3#4#5#6#7{\relax
  \ifnum #1>\c@tocdepth 
  \else
    \par \addpenalty\@secpenalty\addvspace{#2}%
    \begingroup \hyphenpenalty\@M
    \@ifempty{#4}{%
      \@tempdima\csname r@tocindent\number#1\endcsname\relax
    }{%
      \@tempdima#4\relax
    }%
    \parindent\z@ \leftskip#3\relax \advance\leftskip\@tempdima\relax
    #5\leavevmode\hskip-\@tempdima #6\nobreak\relax
    ,~#7\par
    \endgroup
  \fi}
\makeatother

\begin{document}
\begin{abstract}   
We prove that any compact complex manifold with finite fundamental group and algebraic dimension zero admits no holomorphic affine connection or holomorphic conformal structure.
\end{abstract}

\maketitle
\begin{center}
\begin{varwidth}{\textwidth}
\tableofcontents
\end{varwidth}
\end{center}

\section{Introduction}

We conjecture that any compact complex manifold of finite fundamental group with a holomorphic Cartan geometry is isomorphic to the model of its Cartan geometry, a homogeneous bundle of complex tori over a rational homogeneous variety.
Roughly, the dynamics of the fundamental group on the universal cover form the essential ingredient in the classification of holomorphic Cartan geometries, so trivial dynamics gives a trivial geometry.
As a first step, we try to study this question in the extreme case of manifolds of algebraic dimension zero, where there are few tools available from algebraic geometry.

A holomorphic affine connection is a holomorphic connection on the holomorphic tangent bundle of a complex manifold.
A compact K\"ahler manifold admits a holomorphic affine connection just when it has a finite holomorphic unramified covering by a complex torus \cite{Inoue/Kobayashi/Ochiai:1980}. In this case the holomorphic affine connection pulls back to the complex torus to a translation invariant affine connection.

Nevertheless, some interesting compact complex manifolds which are not K\"ahler admit holomorphic affine connections. Think, for example, of the Hopf manifold associated to a linear contraction of complex Euclidean space. Also the parallelizable manifolds $G/   \Gamma$ associated to a complex Lie group $G$ and a lattice $\Gamma \subset G$ admit holomorphic affine connections. In~\cite{Ghys:1995} Ghys constructs exotic deformations of quotients  $SL(2,\mathbb{C})/ \Gamma$ which are nonparallelizable complex manifolds   admitting holomorphic affine connections. Moreover Ghys' quotients do not admit nonconstant meromorphic functions.

Our conjecture predicts that holomorphic affine connections do not exist on compact complex manifolds of finite fundamental group.
We prove this here under the additional hypothesis that our manifold's meromorphic functions are constant.

\section{Main results}\label{section:Notations and main result}

The \emph{algebraic dimension} of a complex manifold \(M\) is the transcendence degree of the field of meromorphic functions of \(M\) over the field of complex numbers.
A compact complex manifold \(M\) has algebraic dimension zero just when every meromorphic function on \(M\) is constant.

 The main theorem of the article is the following.

\begin{theorem}\label{theorem:main}
No compact complex manifold with finite fundamental group and algebraic dimension zero admits a holomorphic affine connection.
\end{theorem}

The principal ingredient of the proof  is the following result which might be of independent interest.

\begin{theorem}\label{theorem:invariance}
If a complex abelian Lie algebra   acts holomorphically on a complex manifold \(M\) with a dense open orbit preserving a holomorphic affine connection, then it also preserves a \emph{flat} torsion-free holomorphic affine connection.
\end{theorem}

The same proof  yields the obvious analogue  of Theorem~\ref{theorem:invariance} in the real analytic category, but the smooth category remains a mystery.   The more general result, for Lie algebras which might not be abelian, is not true in higher dimension: the canonical action of $SL(2,\mathbb{C})$ on $SL(2,\mathbb{C} ) / \Gamma$ preserves the standard connection for which the right-invariant vector fields are parallel, but there are no torsion-free flat affine connections on $SL(2,\mathbb{C} )/ \Gamma$~\cite{Dumitrescu:2009}. Nevertheless, we conjecture that a finite dimensional complex Lie algebra acting holomorphically with a dense open orbit and preserving a holomorphic affine connection always preserves a {\it locally homogeneous} holomorphic affine connection. In the real analytic category and for surfaces this  result was proved in~\cite{Dumitrescu/Guillot:2013}.   

Another result of this article deals with holomorphic conformal structures:

\begin{theorem}  \label{theorem:conformal structures}  No compact complex manifold with finite fundamental group and algebraic dimension zero admits a holomorphic conformal structure.
\end{theorem}

The article is structured as follows.  
Section~\ref{section:algebraic dimension zero} describes the notion of Gromov's rigid geometric structure and shows that compact complex simply connected  manifolds of algebraic dimension zero admit holomorphic rigid geometric structures if and only if they admit toroidal structures (see Proposition~\ref{proposition:rigid} and Proposition~\ref{proposition:toroidal}).
In Section~\ref{section:invariance} several lemmas prove the fact that toroidal actions preserving a holomorphic affine connection always preserve a {\it flat} holomorphic affine connection (Theorem~\ref{theorem:invariance}). Since compact complex simply connected manifolds do not admit flat holomorphic affine connections, this yield to the proof of Theorem~\ref{theorem:main}. Section~\ref{section:conforme} deals with holomorphic conformal structures and contains the proof of Theorem~\ref{theorem:conformal structures}. Eventually, Section~\ref{section:conclusions} presents some conclusions and open questions related to the main results of the article.

\section{Holomorphic geometry in algebraic dimension zero} \label{section:algebraic dimension zero}

Holomorphic affine connections  are geometric structures of algebraic type  which are rigid in Gromov's sense; see the nice expository survey \cite{DAmbra/Gromov:1991} for the precise definition. Roughly speaking the rigidity  comes from the fact
that local biholomorphisms fixing a point and  preserving a connection linearize in exponential coordinates, so they are completely determined by their differential at the fixed point. More generally, the local biholomorphisms preserving a rigid geometric structure are completely determined by a finite jet~\cite{DAmbra/Gromov:1991}.  Also Gromov noticed  that  all known geometric structures are of algebraic kind, in the sense that the natural action of jets of local bihomorphisms on the jets of the geometric structure is algebraic.

So holomorphic affine connections and holomorphic  parallelizations of the holomorphic tangent bundle are examples of rigid holomorphic geometric structures of algebraic type in Gromov's sense. Other important examples of holomorphic rigid geometric structures of algebraic type  are holomorphic Riemannian metrics and holomorphic conformal structures. They will be defined in section \ref{section:conforme}.

In this section we will use the  result obtained by the first author in~\cite{Dumitrescu:2011} (see Corollary 2.2 on page 35) asserting that on complex manifolds with algebraic dimension zero,
rigid  meromorphic  geometric structures of algebraic type are locally homogeneous away from a nowhere dense analytic subset (see also \cite{Dumitrescu:2001b}).  This means that the set of  local holomorphic vector fields preserving a rigid meromorphic geometric structure is transitive away from a nowhere dense  analytic subset.

The same  result was  also proved to be true for holomorphic Cartan connections with algebraic model in~\cite{Dumitrescu:2010b}. In particular, the result applies to holomorphic affine connections on manifolds with algebraic dimension zero. Here we give a more precise result.

A \emph{toroidal structure}, or \emph{toroidal action}, on a complex manifold \(M\) of complex dimension \(n\defeq \dim_{\C{}}{M}\) is a holomorphic group action of the \emph{toroidal group} \(\pr{\C{*}}^n\) on \(M\) with a dense open orbit. 

Two toroidal structures given by two holomorphic actions $G_1$ and $G_2$ of \(\pr{\C{*}}^n\) on $M$ are isomorphic if there exists a biholomorphism of $M$ conjugating the $G_1$-action with the $G_2$-action.

\begin{example}
We thank Misha Verbitsky for suggesting to us the following simple example of toroidal  simply connected compact manifolds of algebraic dimension zero. This example is constructed by deformation of the standard complex structure on a simply connected Calabi-Eckmann manifold and it is a very  particular case of the family of  moment-angle  manifolds constructed in \cite{Panov:2012, Panov:2013aa} as a generalization of the LMVB manifolds. Recall that the LMVB manifolds, constructed and studied  by Lopez de Medrano, Meersseman, Verjovsky and Bosio in  \cite{Bosio:2001, MB:2006, Meersseman:2000} (see also \cite{Lescure/Meersseman:2002,Winkelmann:2008, Ishida:2013aa})   are toroidal  and many of them admit complex affine structures. In \cite{Panov:2013aa} the authors prove that generic  moment-angle manifolds are of algebraic dimension zero.

Consider  the quotient of $(\mathbb{C}^2 \setminus \{ 0 \} ) \times (\mathbb{C}^2 \setminus \{0 \} )$  by  the $\mathbb{C}$-action given by the one-parameter group 
 $\left(  \begin{array}{cc}
                                                                 e^t   &   0\\
                                                                 0     &  e^{\alpha t} \\
                                                                 \end{array} \right)$, with $\alpha \in \mathbb{C} \setminus \mathbb{R}$. This action is holomorphic, proper and free so the quotient is a complex manifold.

The embedding  of $S^3$ as the unit sphere  in $\mathbb{C}^2$ shows that the quotient $M$ is diffeomorphic to $S^3 \times S^3$. This complex structure on $S^3 \times S^3$ fibers over $P^1(\mathbb{C}) \times P^1(\mathbb{C})$. Deform the  $\mathbb{C}$-action to an action of a generic one-parameter subgroup  in $GL(2, \mathbb{C}) \times GL(2, \mathbb{C})$. Then the corresponding complex structure on $M$ is of algebraic dimension zero, as proved in \cite{Panov:2013aa}. Moreover, since the $\mathbb{C}$-action is diagonalizable, its centralizer in  $GL(2, \mathbb{C}) \times GL(2, \mathbb{C})$ is  isomorphic to \(\pr{\C{*}}^4\). This induces a holomorphic \(\pr{\C{*}}^3\)-action with an open dense orbit on the quotient $M$: a toroidal structure.
\end{example}

The \emph{characteristic subvariety} \(S \subset M\) of a holomorphic abelian Lie algebra action with dense open orbit is the complement of the dense open orbit.
If we write out vector fields \(\Set{v_a}\) spanning the Lie algebra, then \(S\) is precisely the set where
\[
0 = v_1 \wedge v_2 \wedge \dots \wedge v_n.
\]
In particular, \(S \subset M\) is a closed complex hypersurface representing the anticanonical bundle of \(M\).

\begin{proposition}\label{proposition:rigid}
A holomorphic action of an abelian complex Lie algebra is a holomorphic rigid geometric structure if and only if it has a dense open orbit.
\end{proposition}
\begin{proof}
Near a generic point of our complex manifold, the action is a local holomorphic parallelism, so locally rigid.
Near a generic point \(s_0 \in S\) of the characteristic variety, there is a unique vector field \(v_1\) (up to constant rescaling) from our Lie algebra vanishing at that point.
The order of vanishing of \(v_1\) is invariant under the Lie algebra action, because the Lie algebra is abelian.
In a basis \(v_1,v_2,\dots,v_n\) of the Lie algebra, the other vector fields \(v_2, \dots, v_n\) are linearly independent at \(s_0\), spanning \(T_{s_0} S\).

If \(v_1\) vanishes to first order along \(S\) near \(s_0\), then we pick local holomorphic coordinates \(z_1, z_2, \dots, z_n\) in which \(v_1=z_1 \p{z_1}\) and \(v_j=\p{z_j}\) for \(j=2,3,\dots,n\).
Then the holomorphic affine connection \(\nabla=d\) in those coordinates is an invariant of the Lie algebra action, so the action is rigid.

If \(v_1\) vanishes to second order along \(S\) near \(s_0\), then we pick local holomorphic coordinates \(z_1, z_2, \dots, z_n\) in which \(v_1=z_1^2 \p{z_1}\) and \(v_j=\p{z_j}\) for \(j=2,3,\dots,n\).
The holomorphic projective connection \([\nabla]=[d]\) in those coordinates is an invariant of the Lie algebra action, so the action is rigid.

Suppose that \(v_1\) vanishes to some higher order \(q\) along \(S\) near \(s_0\).
Locally quotient by the action of \(v_2, \dots, v_n\) to get a vector field \(v_1\) in one variable, say \(v_1=z^q \p{z}\).
Let \(w=z^{q-1}\) and compute that \(v_1=(q-1)w^2 \p{w}\).
So we first map from coordinates \(z^1,\dots,z^n\) to variables \(w,z^2,\dots,z^n\).
Then our vector fields drop to vector fields in those coordinates.
The pullback of the projective connection is not a projective connection anymore, but it is an invariant geometric structure.
If we look at all local choices of complex functions \(z\) invariant under \(v_2,\dots,v_n\), for which \(v_1=z^q \p{z}\), we see that they are unique up to replacing by \(Z=Z(z)\) for any holomorphic \(Z(z)\) with \(Z(0)=0\) and 
\[
Z^{q-1}=\frac{z^{q-1}}{1+cz^{q-1}},
\]
for any constant \(c \in \C{}\).
But this then determines \(w=z^{q-1}\) up to projective transformation.
So near each smooth point of the characteristic variety, there is a ``multivalued quotient projective connection''.
At order \(q+1\) and higher, the formal automorphisms of the structure have Taylor coefficients determined by lower order Taylor coefficients, local rigidity.

At an arbitrary point, the equations on Taylor coefficients of a formal local biholomorphism which force it to be an automorphism of our vector fields up to a given order are linear in highest derivatives.
Near the generic point of \(S\), these linear equations, at some order, suffice to determine each coefficient as a function of lower order coefficients. 
So these linear equations can only fail to determine coefficients inductively when we sit at some point of \(S\) which lies in a higher codimension stratum.
But then by Hartogs extension, these equations continue to determine coefficients at those points as well, as the rank of the linear expression in the highest derivatives cannot drop.
\end{proof}

If a toroidal action on compact complex manifold \(M\) has a fixed point then \(M\) is a toric projective variety \cite{Ishida/Karshon:2012}.

A by-product of Bochner's theorem is  the fact that any toroidal action with reductive stabilizer at every point is locally linearizable. 
This is not true for general holomorphic $\mathbb{C}^n$-actions  with an open dense  orbit. Think
of the standard  $\mathbb{C}^n$-action by translations on $P^n(\mathbb{C})$ in the neighborhood of points situated on the divisor at infinity. Nevertheless Theorem \ref{theorem:invariance} can be seen as a global  linearization result for $\mathbb{C}^n$-actions preserving a holomorphic affine connection and admitting an open dense orbit.

Let us contrast with non-abelian locally free holomorphic actions admitting an open dense orbit. Guillot classified in \cite{Guillot:2007} holomorphic equivariant compactifications of  quotients of $SL(2, \mathbb{C})/ \Gamma$, where $\Gamma$ is a discrete subgroup in $SL(2, \mathbb{C})$. Unlike  toroidal  actions, those $SL(2, \mathbb{C})$-actions are nonrigid holomorphic geometric structures  of algebraic type; they are rigid only on the open dense orbit $U$ of the $SL(2, \mathbb{C})$-action. Indeed, here the local symmetries of the $SL(2, \mathbb{C})$-action on $U$ pull-back to $SL(2,\mathbb{C})$ as left-invariant  vector fields: they generate right translations in $SL(2, \mathbb{C})$ and  generically they do not extend to all of the  compactification space. The extendability of local symmetries does not hold for those  holomorphic geometric structures: they must be nonrigid by a result of Gromov (\cite{DAmbra/Gromov:1991}, p. 73, 5.15).

\begin{proposition}  \label{proposition:toroidal}
Suppose that \(M\) is a compact, connected and simply connected complex manifold of  complex dimension \(n\) and of algebraic dimension zero.
Then \(M\) admits a holomorphic rigid geometric structure  of algebraic type if and only if \(M\) admits a toroidal structure.
The toroidal structure is then unique.
The toroidal group is a cover of  the identity component \(G\) of the biholomorphism group of \(M\). 
The toroidal action preserves \emph{all} meromorphic  geometric structures of algebraic type  on \(M\).

Every bimeromorphism of \(M\) is a biholomorphism of the open toroidal orbit.
The bimeromorphism group of \(M\) is a semidirect product \(G \rtimes \Gamma\) where \(\Gamma\) is the discrete subgroup of the bimeromorphism group fixing some point \(m_0 \in M\) of the dense open toroidal orbit.
Each element of \(\Gamma\) is determined by its action on \(T_{m_0} M\), giving an injective morphism of Lie groups \(\Gamma \to \GL{T_{m_0} M} \cong \GL{n,\C{}}\). 
The complement in \(M\) of the open toroidal orbit is a complex hypersurface containing at least \(n\) analytic components.
\end{proposition}
\begin{proof}
By the main theorem in \cite{Dumitrescu:2011} (see also~\cite{Dumitrescu:2010b})  any holomorphic rigid geometric structure \(g\) of algebraic type on $M$ is locally homogeneous on an open dense set (away from a nowhere dense analytic subset $S$), meaning that the Lie algebra of  local holomorphic vector fields on $M$ preserving 
$g$ is transitive on an open dense set in $M$. Moreover, since $M$ is simply connected, by  a  result due to Nomizu~\cite{Nomizu:1960} and generalized by Amores~\cite{Amores:1979} and then by Gromov~\cite{DAmbra/Gromov:1991} (p. 73, 5.15)
these local vector fields preserving $g$ extend to all of $M$ to form a finite dimensional complex Lie algebra $\LieG$ of  (globally defined)  holomorphic vector fields $v_i$ acting with a dense open orbit in $M$ and preserving $g$.

Now put together $g$ and the $v_i$ to form another rigid holomorphic geometric structure of algebraic type  $g'=\pr{g,v_i}$. (See \cite{DAmbra/Gromov:1991} for details about the fact that the juxtaposition of a rigid geometric structure with another geometric structure is still a rigid geometric structure in Gromov's sense.) Considering $g'$ instead of $g$ and repeating the same proof as before, the complex Lie algebra $\LieG'$ of those holomorphic vector fields preserving $g'$ acts with a dense open orbit in $M$. But preserving $g'$ means preserving both $g$ and the $v_i$. Hence $\LieG'$ lies in  the center  of $\LieG$.  In particular $\LieG'$ is a complex abelian Lie algebra acting with a dense open orbit in $M$. 
At each point \(m_0\) in the open \(\LieG'\)-orbit, the values \(v(m)\) of the vector fields \(v \in \LieG'\) span the tangent space to the \(\LieG'\)-orbit, i.e. span the tangent space. 
Any linear relation between the values \(v(m)\) of the vector fields \(v \in \LieG'\) is \(\LieG'\)-invariant, so holds throughout the open orbit, and so holds everywhere.
Therefore all vector fields in \(\LieG'\) are linearly independent at every point of the open orbit of \(\LieG'\).

Pick a basis \(v_1, v_2, \ldots, v_n \in \LieG'\).
Notice that any  holomorphic vector field commuting with all \(v_i\), for \(i \in \{1, \ldots, n \}\), is a constant coefficient linear combination of those \(v_i\): this is true on the open \(\LieG'\)-orbit and, consequently, on all of \(M\).  It follows that the centralizer of the Lie algebra \(\LieG'\) in the Lie algebra of all holomorphic vector fields on \(M\) is exactly  \(\LieG'\).
This implies  that \(\LieG'=\LieG\) and  hence  \(\LieG\) is a complex abelian Lie algebra of  dimension \(n\) acting with a dense open orbit.

The proof is the same if we replace $g'$ by the extra rigid  meromorphic  geometric structure  of algebraic type $g''$ on $M$ which is the juxtaposition of $g'$, with any meromorphic geometric structure of algebraic type globally defined on $M$.
This implies that any meromorphic geometric structure on $M$ is \(\LieG\)-invariant.

Therefore all meromorphic vector fields on \(M\) belong to \(\LieG\).
If a connected Lie group acts by bimeromorphisms, then its Lie algebra acts by meromorphic vector fields, so as a Lie subalgebra of \(\LieG\).

Define meromorphic 1-forms \(\omega^a\) on \(M\) by
\[
\omega^a(\xi) \defeq 
\frac{v_1 \wedge v_2 \wedge \dots \wedge v_{a-1} \wedge \xi \wedge v_{a+1} \wedge \dots \wedge v_n}
{v_1 \wedge v_2 \wedge \dots \wedge v_n}.
\]

Take a bimeromorphic map \(f \colon M \to M\).
Then \(f^* \omega^a = h^a_b \omega^b\), for some meromorphic functions \(h^a_b\) forming an invertible matrix.
But meromorphic functions on \(M\) are constant, so \(f\) acts on \(\LieG\) by a linear transformation.
In particular, \(f\) extends to be defined on the dense open \(\LieG\)-orbit, which is therefore invariant under bimeromorphism.
If \(f\) acts trivially on \(\LieG\), then we can pick any point \(m_0\) in the dense open \(\LieG\)-orbit and we can find some element \(e^v \in e^{\LieG}\) in the biholomorphism group of \(M\) which takes \(m_0\) to \(f\of{m_0}\), so we can write \(f=e^v g\), so that \(g\of{m_0}=m_0\) for some bimeromorphism \(g\) of \(M\).
But then \(g\) is uniquely determined by its action on \(\LieG\), i.e. on the tangent space \(T_{m_0} M\), since it commutes with exponentiation of the vector fields, so we have a unique decomposition of the bimeromorphism group of \(M\) into a semidirect product \(G \rtimes \Gamma\) where \(\Gamma \subset \GL{\LieG}=\GL{T_{m_0} M}=\GL{n,\C{}}\) and \(G=e^{\LieG}\). Notice that  compactness of $M$ is only required here in order to ensure
completeness of vector fields  in $\LieG$.

Every meromorphic differential form on \(M\) is closed because it is $\LieG$-invariant. In particular, every meromorphic 1-form on $M$ is closed,  while only \(0\) is exact.
The Albanese dimension of \(M\) is zero.
The indefinite integral of any meromorphic 1-form over the simply connected manifold \(M\) is a meromorphic function on some covering space of the complement of the simple poles of the 1-form.
Therefore every meromorphic 1-form has a simple pole on some component of \(S\).

Let \(\Delta \defeq \homology{1}{M-S,\Z{}}=\homology{1}{G,\Z{}}\).
Then \(\Delta \subset \LieG\) and \(G=\LieG/\Delta\).
Pair \(\gamma \in \Delta, \omega \in \LieG^* \mapsto \int_{\gamma} \omega \in \C{}\).
If this vanishes for some \(\omega\), for every \(\gamma\), then \(\omega\) integrates around each component of \(S\) to a meromorphic function.
But then \(\omega=0\).
Therefore the pairing is nondegenerate.
Define an injection \(\gamma \in \Delta \mapsto v_{\gamma} \in \LieG\) by \(\Delta \subset \LieG\).
Then \(\Delta\) spans \(\LieG\) over \(\C{}\).
Therefore \(\Delta\) contains a complex basis of \(\LieG\), so the action of \(\LieG\) drops to an action of \(\pr{\C{*}}^n\).
If some meromorphic 1-form integrates to zero around each analytic component of \(S\), then its indefinite integral is meromorphic.
So picking one \(\gamma \in \Delta\) around each analytic component of \(S\), the associated \(v_{\gamma} \in \LieG\) already span \(\LieG\).
Therefore the number of analytic components of \(S\) is at least \(n\).
\end{proof}

\section{Abelian groups preserving holomorphic connections}\label{section:invariance}

Suppose that \(M\) is a  complex manifold and \(A\) is an abelian Lie algebra of holomorphic vector fields, acting with a dense open orbit (or equivalently, acting with an open orbit on every component of \(M\)) and preserving a holomorphic connection \(\nabla\) on \(TM\).
In this section we will prove theorem~\vref{theorem:invariance}: that \(A\) also preserves a flat holomorphic affine connection.
Every holomorphic affine connection induces a unique torsion-free affine connection with the same geodesics (see for instance~\cite{Inoue/Kobayashi/Ochiai:1980}). 
We assume that \(\nabla\) is torsion-free and that \(M\) is connected without loss of generality.

The complement of the open orbit is a complex hypersurface \(S \subset M\), possibly singular, where the vector fields \(v \in A\) are not linearly independent, representing the anticanonical line bundle.
If this hypersurface is empty, then in terms of any basis \(v_a \in A\),
\[
\bar{\nabla} v_a \pr{X^b v_b} \defeq \pr{\LieDer_{v_a} X^b} v_b.
\]
is an invariant flat  torsion-free connection.
So we assume that \(S\) is not empty. Theorem ~\vref{theorem:invariance} will be proved through the following sequence of Lemmas.

\begin{lemma}  \label{lemma:first}
We make \(A\) into a commutative, perhaps nonassociative, algebra by \(vw\defeq\nabla_v w\).
\end{lemma}
\begin{proof}
Take any basis \(\Set{v_a} \subset A\).
Then on \(M-S\), these \(\Set{v_a}\) form a parallelism, so we can write
\[
\nabla_{v_a} v_b = \Gamma^c_{ab} v_c
\]
for unique functions \(\Gamma^c_{ab}\).
By \(A\)-invariance, \(\Gamma^c_{ab}\) are constants.
By continuity, the relation \(\nabla_{v_a} v_b = \Gamma^c_{ab} v_c\) continues to hold everywhere on \(M\).
The torsion of \(\nabla\) is
\begin{align*}
\nabla_v w - \nabla_w v - [v,w],
\end{align*}
so by torsion-freedom \(vw=wv\).
\end{proof}

\begin{lemma}
Pick a point \(s_0 \in S\) and let 
\[
I_{s_0}\defeq \Set{v \in A|v\of{s_0}=0} \subset A.
\]
Then \(I_{s_0} \subset A\) is an ideal.
\end{lemma}
\begin{proof}
If \(v \in I_{s_0}\) and \(w \in A\) then clearly \(\nabla_v w\) vanishes wherever \(v\) vanishes.
In other words, \(I_{s_0} \subset A\) is an ideal.
\end{proof}

\begin{lemma} \label{lemma:foliation}
Take an ideal \(I \subset A\).
For each point \(m \in M\), let
\[
I(m)\defeq\Set{v(m)|v \in I} \subset T_m M.
\]
There is a unique nowhere singular holomorphic foliation \(F_I\) on \(M\) so that the tangent spaces of the leaves are
\[
T_m F_I = I(m)
\]
for every \(m \in M-S\).
The normal bundle of \(F_I\) is flat along each leaf of \(F_I\), with trivial leafwise holonomy.
\end{lemma}
\begin{proof}
Every tangent vector to \(M-S\) is \(v\of{m}\) for a unique \(v \in A\).
For any \(w \in I\), the vector field \(\nabla_v w\) lies in \(I\), so \(\nabla_{v(m)} \colon I(m) \to I(m)\).
In other words, through every point \(m \in M-S\), the subspace \(I(m)\) is invariant under parallel transport in every direction.
Therefore the orbit of the vector fields in \(I\) through any such point is totally geodesic.
Parallel transport extends the subspaces \(I(m)\) to be defined at every point \(m \in M\), continuously and therefore holomorphically, \emph{but} \(I(m) \ne \Set{v(m)|v \in I}\) for any \(m \in S\).
These subspaces \(I(m)\) are the tangent spaces to the \(I\)-orbits throughout \(M-S\), and therefore form a bracket closed subbundle of the tangent bundle on \(M-S\).
This ensures bracket closure on all of \(M\) of the subbundle spanned by these \(I(m)\) by continuity.
Let \(F=F_I\) be the associated holomorphic totally geodesic foliation of \(M\) whose tangent spaces are the spaces \(I(m)\).
Let \(\pi \colon TM \to TM/TF\) be the quotient to the normal bundle of the foliation.
If \(w \in A\) and \(v \in I\) then we let \(n\defeq \pi \circ w\) be the associated quotient section of the normal bundle, and compute
\begin{align*}
\nabla_v n 
&=
\nabla_v \pi \circ w,
\\
&=
\pi \circ \nabla_v w,
\\
\intertext{because \(\pi\) is parallel, as the foliation is totally geodesic,}
&=
0
\end{align*}
because \(\nabla_v w \in I\) is tangent to the foliation \(F\).
Therefore the quotient sections \(n=\pi \circ w\) of the normal bundle of the foliation \(F\) are parallel along the leaves of \(F\).
So the quotient bundle \(TM/TF\) is flat along the leaves, with trivial leafwise holonomy, wherever \(A/I \mapsto A(m)/I(m)\) is injective, and in particular on \(M-S\).
But this is a dense open set, so \(TM/TF\) is flat along every leaf, with trivial leafwise holonomy.
\end{proof}

\begin{lemma}
For any choice of basis \(\Set{v_a} \subset A\), define meromorphic 1-forms \(\omega^a\) on \(M\) by
\[
\omega^a(\xi) \defeq 
\frac{v_1 \wedge v_2 \wedge \dots \wedge v_{a-1} \wedge \xi \wedge v_{a+1} \wedge \dots \wedge v_n}
{v_1 \wedge v_2 \wedge \dots \wedge v_n}.
\]
These are the dual 1-forms to our basis of \(A\), i.e. \(\omega^a\of{v_b}=\delta^a_b\).
In particular, take a point \(s_0 \in S\) and pick our basis \(\Set{v_a} \subset A\) by first picking some basis \(\Set{v_i} \subset I_{s_0}\) and then picking a maximal set of vector fields \(\Set{v_{\mu}} \subset A\) taking on linearly independent values in \(T_{s_0} M\).
Then \(\omega^{\mu}\) are holomorphic and linearly independent 1-forms near \(s_0\).
\end{lemma}
\begin{proof}
Clearly \(\omega^i\) is singular at \(s_0\) because \(\omega^i\of{v_j}=1\) everywhere, including at \(s_0\), while \(v_j\of{s_0}=0\).
On the other hand, it is less clear whether \(\omega^{\mu}\) are holomorphic near \(s_0\).
Let \(I=I_{s_0}\) and \(F=F_I\).
We can equivalently define the \(\omega^{\mu}\) near \(s_0\) by the linear holomorphic equations: 
\begin{enumerate}
\item
\(\omega^{\mu}=0\) on \(TF\) and
\item
\(\omega^{\nu}\of{n_{\nu}}=\delta^{\mu}_{\nu}\), where \(n_{\nu}\defeq\pi\circ v_{\nu}\) is the associated parallel basis of the normal bundle of the foliation \(F\).
\end{enumerate}
In particular, since these holomorphic linear equations specify the \(\omega^{\mu}\), these \(\omega^{\mu}\) are holomorphic everywhere near \(s_0\).
\end{proof}

\begin{lemma}
For every nontrivial  \(v \in A\), the zero locus of \(v\) is a union of complex analytic  irreducible components of \(S\). 
\end{lemma}
\begin{proof}
Pick some \(v \in A\) and some point \(s_0 \in S\) so that \(v\of{s_0} \ne 0\).
We can pick a basis \(\Set{v_i,v_{\mu}}\subset A\) as in the previous lemma, so that \(v\) is one of the \(v_{\mu}\).
By the Hartogs extension theorem, the associated \(\omega^{\mu}\) as defined in the previous lemma are holomorphic on \(M\) except on certain hypersurfaces.
These hypersurfaces lie inside \(S\), because away from \(S\) all of the \(\omega^a\) are holomorphic.
So these hypersurfaces form a union of certain complex analytic irreducible components of \(S\) not passing through \(s_0\).
By the previous lemma, these hypersurfaces coincide with the zero locus of \(v=v_{\mu}\). 
Therefore the zero locus of  \(v\) is a union of complex analytic irreducible components of \(S\) not passing through  \(s_0\). In particular, in the neighborhood of a vanishing point of \(v\) which is
a smooth point of \(S\), the zero locus of \(v\) is exactly \(S\).
\end{proof}

\begin{lemma}
For any smooth point \(s_0 \in S\), the ideal \(I_{s_0} \subset A\) is principal.
In other words there is a vector field \(v \in A\) so that \(v\) vanishes at every point on the complex analytic irreducible component of \(S\) through \(s_0\), \(v\) doesn't vanish at the generic point of \(M\), and any element of $A$ vanishing at \(s_0\)  is a  constant multiple of \(v\).
\end{lemma}
\begin{proof}
In geodesic local coordinates of any connection, any Killing field of the connection is linearized; this a very classical result in the field, see for instance \cite[p. 6 lemma 7]{Dumitrescu/Guillot:2013}.
Therefore every \(v \in I_{s_0}\) vanishing at \(s_0\) has nonzero linearization at \(s_0\) or is \(v=0\).
The commuting of all of the \(v \in A\) ensures that these linearizations commute.
Write out a basis of these vector fields \(v \in I_{s_0}\), say as
\[
v_i = C_{ib}^{\phantom{i}a} x^b \pd{}{x^a}.
\]
The matrices \(C_i\) commute and are linearly independent.
The points of \(S\) near the origin in these coordinates are precisely the points where each  \(v_i\) vanishes, i.e. the points \(x\) where \(C_i x=0\), for any one value of \(i\).
Since \(S\) is a complex hypersurface, the kernels of all of the \(C_i\) must be the same complex hypersurface.
Since the kernel of each \(C_i\) is a hypersurface, and each \(C_i\) is a square matrix, each \(C_i\) has 1-dimensional image, so this image is an eigenspace of \(C_i\). 
Each \(C_i\) preserves the eigenspaces of the others, so they all share the same eigenvectors.
Hence all of the \(C_i\) are scalar multiples of one another.
If there is more than one of these \(C_i\) matrices, then they are not linearly independent.
\end{proof}

\begin{lemma} \label{lemma:flatness}
Let \(I \subset A\) be the ideal generated by all vector fields \(v \in A\) which vanish at some point of \(M\).
The leaves of the foliation \(F_I\) are totally geodesic and the holomorphic affine connection is flat along these leaves.
In particular, if \(I=A\), then the holomorphic affine connection is already flat everywhere.
\end{lemma}
\begin{proof}
Take a maximal collection \(\Set{v_i} \subset A\) of nonzero vector fields, so that each vanishes on a different locus from any of the others.
Since the ideal generated by each element \(v_i\) is principal,
\[
v_i w = \alpha_i(w) v_i,
\]
for \(w \in A\), for a unique \(\alpha_i \in A^*\).
Clearly if \(v_i, v_j \in I\) have different vanishing loci,  \(v_i v_j=\alpha_i\of{v_j}v_i=\alpha_j\of{v_i}v_j\), so \(\alpha_i\of{v_j}=0\) if \(i \ne j\).
So we can write \(v_i v_j = \delta_{ij} \lambda_i v_i\) for some \(\lambda_i \in \C{}\).
We calculate the curvature \(R\) along any leaf of any \(F_I\) at a point of \(M-S\):
\begin{align*}
R\of{v_i,v_j}v_k
&=
\nabla_{v_i} \nabla_{v_j} v_k
-
\nabla_{v_j} \nabla_{v_i} v_k
-
\nabla_{\left[v_i,v_j\right]} v_k,
\\
&=
\lambda_i \lambda_j \delta_{ij} \pr{\delta_{jk} - \delta_{ik}}v_k,
\\
&=0.
\end{align*}
\end{proof}

\begin{lemma} \label{lemma:final}
Again let \(I \subset A\) be the ideal generated by all vector fields \(v \in A\) which vanish at some point of \(M\).
Take a basis \(\Set{v_i} \subset I\).
Choose additional elements \(\Set{v_{\mu}} \subset A\) so that \(\Set{v_i,v_{\mu}} \subset A\) is a basis.
Choose a leaf \(L\) of \(F_I\) and a local basis of parallel sections \(w_i\) of \(TL=\left.TF_I\right|_L\) near a point of \(L\).
Note that \(w_i,v_{\mu}\) span \(\left.TM\right|_L\) near that point.
Extend these \(w_i\) nearby on \(M\) by invariance under the \(v_{\mu}\).
Then the connection \(\nabla'\) defined by
\begin{align*}
0 &= \nabla'_{v_{\mu}} v_{\nu}, \\
0 &= \nabla'_{w_i} w_j, \\
0 &= \nabla'_{v_{\mu}} w_j, \\
0 &= \nabla'_{w_i} v_{\nu}
\end{align*}
extends uniquely to a flat torsion-free  holomorphic \(A\)-invariant connection on \(M\).
\end{lemma}
\begin{proof} Since the leaves of \(F_I\) are totally geodesic and flat (see Lemma~\ref{lemma:flatness}) and the vectors  \(v_{\mu}\) preserve $\nabla$, then $\pr{w_i, v_{\mu}}$ is  a  local frame  of commuting  vector fields.
By construction this local  frame is parallel with respect to   \(\nabla'\). Hence, in adapted local holomorphic  coordinates,  \(\nabla'\)  is the  standard   torsion free flat affine connection.
If we change the choice of \(w_i\), we do so only by constant coefficient linear combinations, so \(\nabla'\) is unchanged.

Notice that $\nabla$ and $\nabla'$ agree on the leaves of \(F_I\). The transverse vector fields \(v_{\mu}\) are parallel with respect to  $\nabla'$; this is not necessarily the case with respect to  $\nabla$. 

Since the foliation \(F_I\)  is holomorphic and nonsingular (see Lemma~\ref{lemma:foliation}), $\nabla'$ is holomorphic on all of $M$. Moreover, $\nabla'$ is \(A\)-invariant by construction (since everything in the definition of \(\nabla'\) is).
\end{proof}

Lemmas~\ref{lemma:first} to~\ref{lemma:final} prove Theorem~\ref{theorem:invariance}.

Let us now deduce Theorem~\ref{theorem:main} from Theorem~\ref{theorem:invariance}.

Assume, by contradiction, that a compact complex simply connected manifold   of algebraic dimension zero  $M$ admits a holomorphic affine connection $\nabla$. Since $\nabla$ is a rigid geometric structure, 
Proposition \ref{proposition:toroidal} shows  that $M$ is toroidal.  Theorem \ref{theorem:invariance} implies  that $M$ admits a flat holomorphic affine connection $\nabla'$ (invariant by the toroidal action), hence a complex affine structure (canonically associated to $\nabla'$) . Since $M$ is simply connected, a developing map of the complex affine structure is a local biholomorphism from $M$ to the complex affine space  of the same dimension. But the compactness of $M$ implies that any such  holomorphic map   must be constant: a contradiction. This finishes the proof of Theorem \ref{theorem:main}.

\section{Holomorphic conformal structures} \label{section:conforme}

\begin{definition}A holomorphic Riemannian metric on a complex manifold  $M$  is a holomorphic section \(q\) of the bundle  $S^2(T^{*}M)$  of complex quadratic forms on  $M$ such that in any point  $m$ in $M$ the quadratic form  $q(m)$ is nondegenerate.
\end{definition} 

As in the real Riemannian and pseudo-Riemannian settings, any holomorphic Riemannian metric $q$  on $M$ determines a unique  torsion free holomorphic affine connection with respect to which $q$ is a parallel tensor. Starting with this Levi-Civita connection one computes the curvature tensor of $q$. Recall that $q$ is called flat if its curvature tensor vanishes everywhere. In this case $q$ is locally isomorphic to $dz_1^2+dz_2^2 + \ldots + dz_n^2$ on $\mathbb{C}^n$ and $M$ is locally modelled on $\mathbb{C}^n$ with transition maps in $G=O(n,\mathbb{C}) \ltimes \mathbb{C}^n$.

A holomorphic Riemannian metric on $M$ defines an isomorphism between $TM$ and $T^*M$. Moreover, up to a double cover on $M$, the canonical bundle and the anticanonical bundle of $M$ are holomorphically
trivial (see for instance \cite{Dumitrescu:2011}). Consequently, Proposition \ref{proposition:toroidal} implies that compact complex simply connected manifolds with algebraic dimension zero do not admit holomorphic Riemannian metrics. In complex dimension three this result was proved in \cite{Dumitrescu:2011} (see Corollary 4.1 on page 44) without any hypothesis on the algebraic dimension.\\

A more flexible geometric structure is a holomorphic conformal structure.  

\begin{definition}A holomorphic conformal structure  on a complex manifold  $M$  is a holomorphic section $\omega$ of the bundle  $S^2(T^{*}M) \otimes L$, where $L$ is a holomorphic line bundle over $M$, such that at any point  $m$ in $M$ the section  $\omega(m)$ is nondegenerate.
\end{definition} 

Roughly speaking this means that $M$ admits an open cover such that on each open set in the cover, $M$ admits a holomorphic Riemannian metric, and on the overlaps of two open sets the two given holomorphic Riemannian metrics agree up to a nonzero multiplicative constant.

Here the flat example is the quadric \(z_0^2+ z_1 ^2 + \ldots +z_{n+1}^2=0\) in $P^{n+1}(\mathbb{C})$ with the conformal  structure induced by the quadratic form $dz_0^2+dz_1^2+\ldots+dz_{n+1}^2$ on  the quadric. The automorphism group of the quadric with its canonical conformal structure is $\mathbb{P}O(n+2, \mathbb{C})$. 

A classical result due to Gauss asserts that all conformal structures on surfaces are locally isomorphic  to the two-dimensional quadric.

Any manifold $M$ of complex dimension $n \geq 3$  bearing a flat holomorphic conformal  structure (meaning that the Weyl tensor of curvature vanishes on all of $M$) is locally modelled on the quadric.

Recall that Kobayashi and Ochiai classified in \cite{Kobayashi/Ochiai:1982} the complex compact surfaces locally modelled on the quadric. More recently, Jahnke and Radloff classified projective compact complex threefolds bearing holomorphic conformal structures \cite{Radloff:2005} and also projective compact complex manifolds locally modelled on the quadric \cite{Jahnke:2015aa}.

Let us  prove now  Theorem \ref{theorem:conformal structures}.

\begin{proof}  For surfaces, the result is a direct consequence of the classification given in \cite{Kobayashi/Ochiai:1982}.

We suppose now that our manifold $M$ is of complex dimension at least $3$. Then the holomorphic conformal structure is a rigid geometric structure \cite{DAmbra/Gromov:1991}.

Up to a finite cover, we can assume that  $M$ is simply connected and, as a consequence of Proposition \ref{proposition:toroidal}, $M$ is toroidal. On the open dense orbit $U$ of the toroidal group the holomorphic tangent bundle of $M$ is trivial. In particular, its canonical bundle is trivial and hence the holomorphic conformal structure admits a global representative  on $U$ which is a holomorphic Riemannian metric $q$.  

Denote by $v_1, v_2, \ldots v_n$ the fundamental vector fields of the toroidal action. Since both the conformal structure and the holomorphic  section $v_1 \wedge v_2 \wedge \ldots \wedge v_n$ of the canonical bundle are invariant by the toroidal action, it follows that the holomorphic Riemannian metric $q$ is invariant by the toroidal action. But a holomorphic Riemannian metric invariant by a transitive action of an abelian group is flat. In particular, the conformal structure is flat on $U$ and hence on all of $M$.

If follows that $M$ is locally modelled on the quadric. Since $M$ is simply connected, the developing map of the flat conformal structure is a local biholomorphism from $M$ into the quadric. This is impossible since the quadric is an algebraic manifold and $M$ has algebraic dimension zero.
\end{proof}

\section{Conclusion} \label{section:conclusions}

We conjecture that any compact complex manifold with finite fundamental group bearing a holomorphic Cartan geometry is biholomorphic to the model.
In particular, this implies that compact complex manifolds bearing holomorphic affine connections  have infinite fundamental group. The result is known for K\"ahler manifolds \cite{Inoue/Kobayashi/Ochiai:1980}. Above we prove this fact for manifolds with algebraic dimension zero.

 The conjecture is also open and interesting for the particular case of  holomorphic projective connections. Our approach here could  still work for manifolds with algebraic dimension zero endowed with holomorphic projective connections, if we knew how to prove an equivalent of Theorem~\vref{theorem:invariance} for projective connections.
It seems likely that any holomorphic projective connection invariant under an abelian Lie algebra action with a dense open orbit and  preserving no holomorphic affine connection is flat, implying an analogue of Theorem~\ref{theorem:invariance} for projective connections. 

\bibliographystyle{amsplain}
\bibliography{affine-conn-alg-dim-zero}

\end{document}